\documentclass[a4paper]{amsart}
\usepackage{txfonts,amsmath,amstext,amsthm,amscd,amsopn,verbatim,amssymb,amsfonts}
\usepackage{fullpage}

\usepackage{url}
\usepackage[bbgreekl]{mathbbol}
\usepackage{color}
\usepackage[all]{xy}
\usepackage{tikz}
\usetikzlibrary{matrix}
\usetikzlibrary{shapes}
\usetikzlibrary{arrows}
\usetikzlibrary{calc,3d}
\usetikzlibrary{decorations,decorations.pathmorphing}
\usetikzlibrary{through}
\tikzset{ext/.style={circle, draw,inner sep=1pt},int/.style={circle,draw,fill,inner sep=2pt},nil/.style={inner sep=1pt}}
\tikzset{exte/.style={circle, draw,inner sep=3pt},inte/.style={circle,draw,fill,inner sep=3pt}}
\tikzset{diagram/.style={matrix of math nodes, row sep=3em, column sep=2.5em, text height=1.5ex, text depth=0.25ex}}
\tikzset{diagram2/.style={matrix of math nodes, row sep=0.5em, column sep=0.5em, text height=1.5ex, text depth=0.25ex}}
\theoremstyle{plain}
  \newtheorem{thm}{Theorem}
  \newtheorem{defi}{Definition}

\theoremstyle{definition}
  \newtheorem{ex}{Example}
  \newtheorem{rem}{Remark}

\newcommand{\R}{{\mathbb{R}}}
\newcommand{\K}{{\mathbb{K}}}

\newcommand{\e}{{\mathsf{e}}}
\newcommand{\ee}{{\mathsf{ue}}}

\newcommand{\Ger}{{\mathsf{Ger}}}

\newcommand{\Tw}{\mathit{Tw}}

\newcommand{\op}{\mathcal}

\newcommand{\Br}{\mathsf{Br}}

\newcommand{\Lie}{\mathsf{Lie}}

\newcommand{\hoLie}{\mathsf{hoLie}}

\newcommand{\hoe}{\mathsf{hoe}}

\newcommand{\Conv}{\mathrm{Conv}}

\newcommand{\Com}{\mathsf{Com}}
\newcommand{\uCom}{\mathsf{uCom}}

\newcommand{\bpm}{\begin{pmatrix}}
\newcommand{\epm}{\end{pmatrix}}

\DeclareMathOperator{\End}{End}

\newcommand{\Der}{\mathsf{Der}}

\newcommand{\pLie}{\mathsf{preLie}}
\newcommand{\Pois}{\mathsf{Pois}}

\begin{document}
\title{Triviality of the higher Formality Theorem}

\author{Damien Calaque}
\address{I3M \\ Universit\'e  Montpellier 2\\ Case courrier 051 \\
34095 Montpellier cedex 5, France}
\email{damien.calaque@univ-montp2.fr}
\thanks{D.C. acknowledges the support of the Swiss National Science Foundation (grant $200021\underline{~}137778$).
T.W. acknowledges the support of the Swiss National Science Foundation (grants PDAMP2\_137151 and 200021\_150012).}
\author{Thomas Willwacher}
\address{Department of Mathematics\\ University of Zurich\\ Winterthurerstrasse 190 \\ 8057 Zurich, Switzerland}
\email{t.willwacher@gmail.com}

%\author{Thomas Willwacher}
%\address{Department of Mathematics\\ Harvard University\\ One Oxford Street \\ Cambridge, MA, USA}
%\email{t.willwacher@gmail.com}

%\thanks{The author was partially supported by the Swiss National Science Foundation (grant 200020-105450).}
% \subjclass[2000]{16E45; 53D55; 53C15; 18G55}
% \date{}
%\keywords{Formality, Deformation Quantization, Operads}

\begin{abstract}
It is noted that the higher version of M. Kontsevich's Formality Theorem is much easier than the original one. 
Namely, we prove that the higher Hochschild-Kostant-Rosenberg map is already a $\hoe_{n+1}$-formality quasi-isomorphism whenever $n\geq2$. 
\end{abstract}

\maketitle

\section{Introduction}

Let $A$ be any smooth commutative $\mathbb{K}$-algebra essentially of finite type. 
We may consider $A$ as an associative $\mathbb{K}$-algebra only, say $A_1$. As such we may form its Hochschild cochain complex
\[
C(A_1)=\bigoplus_{k\geq 0}Hom_{\mathbb{K}}(A^{\otimes k},A)[-k]
\]
endowed with the Hochschild differential.
The cohomology of $C(A_1)$ is computed by the Hochschild-Kostant-Rosenberg Theorem, which states that the Hochschild-Kostant-Rosenberg (HKR) map
\[
\Phi_{HKR}\colon S_A\Big(\Der(A)[-1]\Big) \longrightarrow C(A_1)
\]
sending a $k$-multiderivation to the obvious map $A^{\otimes k}\to A$ is a quasi-isomorphism of complexes.
Note that $S_A\Big(\Der(A)[-1]\Big)$ is endowed with the zero differential. 

In fact, the degree shifted complexes $S_A\Big(\Der(A)[-1]\Big)[1]$ and $C(A_1)[1]$ are endowed with differential graded (dg) Lie algebra 
structures, with the Schouten bracket and the Gerstenhaber bracket, respectively. 
The central result of deformation quantization is M. Kontsevich's formality Theorem \cite{K1}, stating that there 
is an $\infty$-quasi-isomorphism of dg Lie algebras
\[
S_A\Big(\Der(A)[-1]\Big)[1] \longrightarrow C(A_1)[1]
\] 
extending the HKR map. 

Actually, $S_A\Big(\Der(A)[-1]\Big)$ also carries the structure of a Gerstenhaber algebra (or, $\e_2$ algebra). 
Kontsevich's result has been strengthened by D. Tamarkin \cite{tamanother}, 
who showed that there also exists an $\infty$-quasi-isomorphism of $\Ger_\infty$ algebras (or, $\hoe_2$ algebras) 
$$
S_A\Big(\Der(A)[-1]\Big) \longrightarrow C(A_1)
$$ 
extending the HKR map, for some choice of $\Ger_\infty$ structure on the right hand side.

There is a natural generalization of the objects involved to the higher setting. First, we may consider the commutative algebra $A$ as an 
$\e_{n}=H_{-\bullet}(E_{n})$ algebra, say $A_{n}$, with trivial bracket. We assume that $n\geq2$. 
We may consider the $\e_{n}$-deformation complex which we also denote by $C(A_{n})$. 
There is a version of the Hochschild-Kostant-Rosenberg Theorem stating that the natural inclusion
\[
\Phi_{HKR}^n\colon S_A\Big(\Der(A)[-n]\Big) \to C(A_{n})
\]
is a quasi-isomorphism of complexes.

There is a natural $\e_{n+1}$ algebra structure on $S_A\Big(\Der(A)[-n]\Big)$, with product being the symmetric product and bracket 
being a degree-shifted version of the Schouten bracket.
Similarly, there is an explicit $\hoe_{n+1}$ structure on $C(A_{n})$, constructed by D. Tamarkin \cite{tamenaction}.
Here $\hoe_{n+1}=\Omega(\e_{n+1}^{\textrm{!`}})$ is the minimal resolution of the operad $\e_{n+1}$, i.~e., the cobar construction of 
the Koszul dual cooperad $\e_{n+1}^{\textrm{!`}}\cong \e_{n+1}^*\{n+1\}$.
The higher formality conjecture states that the (quasi-iso)morphism $\Phi_{HKR}^n$ may be extended to an $\infty$-(quasi-iso)morphism 
of $\hoe_{n+1}$ algebras. 

The content of the present paper is to point out that this conjecture is somehow trivial.
\begin{thm}\label{thm:main}
For $n\geq 2$ the HKR map $\Phi_{HKR}^n\colon S_A\Big(\Der(A)[-n]\Big) \to C(A_{n})$ is already a quasi-isomorphism of $\hoe_{n+1}$ algebras. 
\end{thm}

This result might be known to the experts, but the authors are unaware of any reference. 
The proof boils down to a straightforward direct calculation.

\begin{rem}
As will be clear from the proof the statement of Theorem \ref{thm:main} holds true for $A$ the algebra of smooth functions on a smooth manifold, if one replaces the Hochschild complex by the continuous Hochschild complex, or by the complex of multi-differential operators.   
\end{rem}

\begin{rem}\label{rem2}
Note that there is a choice in the precise definition of the ``Hochschild'' complex $C(A_{n})$, 
essentially depending on a choice of cofibrant model for $\e_{n}$. 
We choose here the minimal model $\hoe_{n}$. 
For some other model $\op P$, solving the higher formality conjecture will be ``as complicated as'' picking a morphism $\op P\to \e_{n}$. 
The higher formality conjecture for that model can then be recovered by transfer, using Theorem \ref{thm:main}.
\end{rem}
\begin{rem}
Theorem \ref{thm:main} remains valid for any differential graded algebra $A$ as soon as one replaces $\mathsf{Der}(A)$ 
by its right derived variant $\mathbb{D}\mathsf{er}(A)$. 
Moreover, functoriality of the HKR map allows us to freely sheafify and get in particular that, for a quasi-projective derived scheme $X$ 
and $n\geq2$, the HKR map 
$$
\Phi_{HKR}^n\colon S_{\mathcal O_X}\Big(\mathbb{T}_X[-n]\Big) \to C\big((\mathcal O_X)_n\big)
$$
is a quasi-isomorphism of sheaves of $\hoe_{n+1}$ algebras (the only subtlety is to make $C\big((\mathcal O_X)_n\big)$ into a 
sheaf\footnote{One shall use the quasi-isomorphic sub-complex of multi-differential operators in $C(A_n)$, which sheafifies well. }). 
Note that this is slightly different from the main result of \cite[Section 5]{To}, also called higher formality, where it is proved that the 
$\e_n$ Hochschid complex of $X$ is weakly equivalent to the $E_n$ Hochschild complex of $X$ as a $\Lie_{n+1}$ algebra (in our context this is 
more or less the content of Remark \ref{rem2}, but then the hard part would be to prove that the $\Lie_{n+1}$ structures on Hochschild complexes 
appearing in the present paper and the ones appearing in \cite{To} are the same). 
\end{rem}
\begin{rem}
Observe that for $n\geq2$ we have $\e_n=\Pois_n$. 
All our results and constructions remain valid for every $n\in\mathbb{Z}$ if one uses $\Pois_n$ in place of $\e_n$. 
\end{rem}

\subsection*{Structure of the paper}

In section \ref{sec:notation} we recall some basic definitions and notation. 
Section \ref{sec:tamconstruction} contains a rewording of D. Tamarkin's construction of the $\hoe_{n+1}$ algebra structure on $C(A_{n})$.
The proof of Theorem \ref{thm:main} is a small direct calculation which is presented in section \ref{sec:theproof}.
%Finally section \ref{} contains a short discussion of the space of choices for the higher formality morphism.

\subsection*{Acknowledgements}

Damien Calaque thanks Bertrand To\"en for explaining him that HKR is true {\it without any assumption} in the derived context 
and for many discussions about the formality for derived schemes and stacks. 

\section{Notation}\label{sec:notation}

We will work over a ground field $\K$ of characteristic 0; all algebraic structures should be understood over $\mathbb{K}$. 
We will use the language of operads throughout. A good introduction can be found in the textbook \cite{lodayval}, from which we freely 
borrow some terminology. 

For a (co)augmented (co)operad $\mathcal O$ we denote by $\mathcal O_{\circ}$ the (co)kernel of the (co)augmentation. It is a pseudo-(co)operad: 
i.~e.~it does not have a (co)unit. 

\subsection{Our favorite operads}

\subsubsection{The $\e_n$ operad}

We will denote by $\e_n$ the homology of the topological operad $E_n$, for every $n\geq 1$. Note that we work with cohomological gradings 
(i.e.~our differentials have degree $+1$), so that homology sits in non-positive (cohomological) degree. 
 
As an example, $\e_1$ is the operad governing non-unital associative algebras. 
For $n\geq2$, the operad $\e_n$ is isomorphic to an operad obtained by means of a distributive law: $\e_n\cong \Com\circ \Lie_n$, 
where $\Lie_n:=\Lie\{n-1\}:=\mathcal S^{1-n}\Lie$ is a degree shifted variant of the $\Lie$ operad. 
In particular, the space $\e_n(N)$ is spanned by formal linear combinations of ``Gerstenhaber words'', like
\[
[X_1,X_2]\cdot X_4 \cdot [X_3,X_5]\,,
\]
in $N$ formal variable $X_1,\dots, X_N$, each occurring once.
We thus have an obvious map $\Lie_n\to \e_n$.

\subsubsection{The $\hoe_n$ dg operad}

The minimal resolution of $\e_n$, resp.~$\Lie_n$, is denoted by $\hoe_n$, resp.~$\hoLie_n$. 
In particular $\hoe_n=\Omega(\e_n^{\textrm{!`}})$ where $\Omega(\cdot)$ denotes the cobar construction and $\e_n^{\textrm{!`}}\cong \e_n^*\{n\}$ 
is the Koszul dual cooperad of $\e_n$. Note that $\hoe_n$ and $\hoLie_n$ are dg operads. 

One may understand elements of $\e_n^*(N)$ by linear combinations of ``co-Gerstenhaber words'', like 
\[
\underline{X_1X_2}\cdot X_4 \cdot \underline{X_3X_5X_6}\,,
\]
in $N$ formal variable $X_1,\dots, X_N$, each occurring once.
The underline shall indicate that one equates linear combinations that correspond to (signed) sums of shuffle permutations to zero.

\medskip

We may also consider the extended $\e_n$ operad $\ee_n=\mathsf{uCom}\circ\Lie_n$, which contains one nullary operation, 
i.e.~$\ee_n(0)=\mathsf{uCom}(0)=\K$. It governs unital $\e_n$-algebras and can be obtained as the homology of the topological 
little disks operad, which has a nullary operation acting by deleting disks.

\subsubsection{The $\pLie$ operad}\label{sssec:pLie}

We will denote by $\pLie$ the operad encoding pre-Lie algebras. Following \cite{CL}, it admits the following combinatorial description. 
We first introduce the set $\mathcal T(I)$ of rooted trees with vertices labelled by a finite set $I$, which is constructed {\it via} the 
following inductive process: 
\begin{itemize}
\item $\mathcal T(\emptyset)$ is empty. 
\item $\mathcal T(\{i\})$ consists of single rooted tree having only one root-vertex labelled by $i$: 
$\begin{tikzpicture}
\node (v0) at (0,-0.15) {};
\node[ext] (v1) at (0,0) {$i$};
\draw (v0) edge (v1);
\end{tikzpicture}$
\item Let $I$ be a finite set, $i\in I$ and a partition $I_1\sqcup\cdots\sqcup I_k=I-\{i\}$. 
Given rooted trees $\mathbf{t}_\alpha\in\mathcal T(I_\alpha)$, $\alpha=1,\dots,k$ one can construct a new rooted tree 
$B_+(\mathbf{t}_1,\dots,\mathbf{t}_k)\in\mathcal T(I)$ by grafting the root of each $\mathbf{t}_\alpha$, $\alpha=1,\dots,k$, on a common 
new root labelled by $i$: 
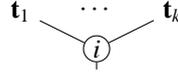
$$\begin{tikzpicture}
\node (v2) at (-1,0.5) {$\mathbf{t}_1$};
\node (v3) at (1,0.5) {$\mathbf{t}_k$};
\node (v0) at (0,-0.15) {};
\node (vv) at (0,0.5) {$\cdots$};
\node[ext] (v1) at (0,0) {$i$};
\draw (v0) edge (v1);
\draw (v2) edge (v1);
\draw (v3) edge (v1);
%\draw (v2) edge[dotted] (v3);
\end{tikzpicture}$$
\end{itemize}
Then $\pLie(I)$ is the vector space generated by $\mathcal T(I)$, and the operadic composition can be defined in the following way: if $J$ 
is another finite set, $i\in I$, $\mathbf{t}\in\mathcal T(I)$ and $\mathbf{t}'\in\mathcal T(J)$, then $\mathbf{t}\circ_i\mathbf{t}'$ is 
described as a sum over the set of functions $f$ from incoming edges at the vertex $i$ of $\mathbf{t}$ to the vertices of $\mathbf{t}'$. 
For any such $f$, the corresponding term is obtained by removing the vertex $i$ from $\mathbf{t}$, reconnecting the outgoing edge to 
the root of $\mathbf{t}'$ and reconnecting the incoming edges $e$ to the vertex $f(e)$. 
The root of the result is taken to be the root of $\mathbf{t}$ if this is not $i$, or the root of $\mathbf{t}'$ otherwise. 

\medskip

Note that for any operad $\mathcal O$, the vector spaces $\prod_{n\geq0}\mathcal O(n)$ and $\prod_{n\geq0}\mathcal O(n)^{S_n}$ are naturally $\pLie$ algebras. 

\medskip

Recall also that there is a morphism of operads $\Lie\to\pLie$ which sends the generator of $\Lie$ to 
$\scriptstyle
\begin{tikzpicture}[scale=.85,baseline=.5]
\node (v0) at (0,-0.15) {};
\node[ext] (v1) at (0,0) {$1$};
\node[ext] (v2) at (0,0.6) {$2$};
\draw (v0) edge (v1);
\draw (v2) edge (v1);
\end{tikzpicture}
~-~
\begin{tikzpicture}[scale=.85,baseline=.5]
\node (v0) at (0,-0.15) {};
\node[ext] (v1) at (0,0) {$2$};
\node[ext] (v2) at (0,0.6) {$1$};
\draw (v0) edge (v1);
\draw (v2) edge (v1);
\end{tikzpicture}$~.
Hence any pre-Lie algebra is also a Lie algebra (obtained by skew-symmetrizing the pre-Lie product). 

\subsection{The Hochschild complex of a $\hoe_n$ algebra}

For a $\hoe_n$ algebra $B$, we define the ``Hochschild'' complex as the degree shifted convolution dg Lie algebra
\[
C(B)=\Conv\big(\ee_n^*\{n\}, \End_B\big)[-n]
\]
where $\End_B$ is the endomorphism operad of $B$ and the differential is the Lie bracket with the element of $C(B)$ corresponding to the $\hoe_n$ structure.
In particular, if $B=A_n$ is as in the introduction, then
\[
C(B) \cong A \oplus \Conv\big(\e_n^*\{n\}, \End_A\big)[-n]. 
\]
as complexes. 

Note that there is a natural inclusion 
\[
\Phi_{HKR}^n\colon S_A\Big(\Der(A)[-n]\Big) \longrightarrow C(A_{n})
\]
whose image consists of the elements in 
\[
\Conv\big(\mathsf{uCom}^*\{n\}, \End_A\big)
\]
that furthermore are (i.e.~take values in) derivations in each slot. Analogously to the usual HKR Theorem one may check the following result.
\begin{thm}[Higher Hochschild-Kostant-Rosenberg Theorem]
\label{thm:higherHKR}
The map $\Phi_{HKR}^n$ is a quasi-isomorphism of complexes for each $n\geq 2$.
\end{thm}
\begin{proof}[Sketch of proof]
One simply observes that, since the bracket on $A$ is zero, we have  
$$
\Conv\big(\ee_n^*\{n\}, \End_A\big)[-n]=S_A\Big(\Conv\big(\Lie^*\{1\}, \End_A\big)[-n]\Big)=S_A\big(\mathbb{D}\mathsf{er}(A)[-n]\big)\,.
$$
If $A$ is smooth and essentially of finite type then the canonical map $\mathsf{Der}(A)\to\mathbb{D}\mathsf{er}(A)$ is a quasi-isomorphism. 
\end{proof}
%Note that the Theorem in fact holds true for every $n$. However, there are some notational difficulties depending on what one denotes by $\e_1$, that we want to disregard.
%Hence we will generally restrict to $n\geq 3$ in the rest of the paper.

\section{A version of D. Tamarkin's proof of the higher Deligne conjecture}\label{sec:tamconstruction}

The goal of this section is to recall D. Tamarkin's proof of the following result.
\begin{thm}[Higher Deligne conjecture, see \cite{tamenaction}]\label{thm:enen1}
%The deformation complex of an $E_n$ algebra carries a natural structure of an $E_{n+1}$ algebra. Concretely, 
The complex $C(A_{n})$ carries a natural $\hoe_{n+1}$ action, given by explicit formulas, for $n\geq 2$.
\end{thm}

\subsection{Braces for a Hopf cooperad}

For any coaugmented cooperad $\op C$ we may define its bar construction $\Omega(\op C)$, which is an operad. 
For example $\hoe_n:=\Omega(e_n^{\textrm{!`}})$.
Here we define a similar construction, the brace construction, which takes a Hopf cooperad $\op C$ and returns an operad $\Br_{\op C}$.

\subsubsection{$\mathcal C$-operads}

In this paragraph we introduce the notion of a $\mathcal C$-operad. %, which one shall think as operads ``parametrized by a Hopf cooperad $\mathcal C$''. 
A $\mathcal C$-operad is an operad $\mathcal O$ such that each $\mathcal O(n)$ carries an $S_n$ equivariant right $\op C(n)$ module structure. We require that furthermore the right module structures are compatible with the operadic compositions, by which we mean that the following diagram commutes:
\[
 \begin{tikzpicture}
  \matrix(m)[diagram]{
  \op O(k)\otimes \op O(n_1)\otimes \cdots \otimes \op O(n_k) \otimes \op C(\sum_j n_j) &  \op O(\sum_j n_j) \otimes \op C(\sum_j n_j)\\
  \op O(k)\otimes \op O(n_1)\otimes \cdots \otimes \op O(n_k) \otimes \op C(k)\otimes \op C(n_1)\otimes \cdots \otimes \op C(n_k) 
  & \\
  \op O(k)\otimes \op O(n_1)\otimes \cdots \otimes \op O(n_k)
  & \op O(\sum_j n_j) \\
  };
  \draw[-latex] (m-1-1) edge (m-1-2) edge (m-2-1) 
  (m-2-1) edge (m-3-1) (m-3-1) edge (m-3-2) (m-1-2) edge (m-3-2);
 \end{tikzpicture}.
\]
Here the two horizontal arrows are the operadic compositions. The upper left vertical arrow is defined using the cooperad structure on $\op C$. The remaining two arrows are defined by using the right action of $\op C(n)$ on $\mathcal O(n)$.
Note that for $\op C=\uCom^*$ a $\op C$-operad is just an ordinary operad.

\begin{ex}
One can check that for an operad $\mathcal P$ and a Hopf cooperad $\mathcal C$, the convolution operad 
$\mathsf{Hom}(\mathcal C\{k\},\mathcal P)$ is naturally a $\mathcal C$-operad for any $k$. Here the right action is obtained by composition with the multiplication on $\op C(n)$ from the right, i.~e.,
\[
 (f\cdot c)(x) =f(cx) 
\]
for $f\in \mathsf{Hom}(\mathcal C\{k\},\mathcal P)(n)$, $c\in \op C(n)$ and $x\in \op C\{k\}(n)$. 
\end{ex}

\subsubsection{The $\pLie_{\mathcal C}$ operad}

In this section we introduce an operad encoding $\mathcal C$-pre-Lie algebras, which are to $\mathcal C$-operads 
what pre-Lie algebras are to operads. For simplicity, we will assume that the Hopf cooperad $\op C$ satisfies: $\op C(0)\cong \op C(1)\cong \K$.
Then one has natural maps 
\begin{equation}\label{equ:extension}
 \op C(j) \to \op C(j+k)\otimes \underbrace{\op C(1)\otimes \cdots \op C(1)}_{j\times}
 \otimes \underbrace{\op C(0)\otimes \cdots \op C(0)}_{k\times} \cong \op C(j+k)
\end{equation}
where the arrow is a cocomposition and the right hand identification uses the canonical identifications $\op C(0)\cong \op C(1)\cong \K$ as algebras.

\begin{ex}
 The most interesting example for us is the Hopf cooperad $\op C=\ee_n^*$, whose $j$-ary cooperations may be interpreted as the cohomology of the configuration space of $j$ points in $\R^n$. There are forgetful maps from the configuration space of $k+j$ points to that of $j$ points, and in this case the extension map \eqref{equ:extension} above is just the pull-back of the forgetful map, forgetting the location of the last $k$ points.
\end{ex}

The operad $\pLie_{\mathcal C}$ consists of rooted trees decorated by a Hopf cooperad $\mathcal C$. Namely, for every finite set $I$, 
$$
\pLie_{\mathcal C}(I):=\bigoplus_{\mathbf{t}\in\mathcal T(I)}
\left(\bigotimes_{i\in I}\mathcal C(\mathbf{t}_i)\right)\,,
$$
where $\mathbf{t}_i$ is the set/number of incoming edges at the vertex labelled by $i$. 

The operadic structure on the underlying trees is the one described in section \ref{sssec:pLie}. 
Let us now explain what happens to the decoration when doing the partial composition $\circ_i$. 
Borrowing the notation from \S\ref{sssec:pLie}, for every $f$ we apply a cooperation:\footnote{Note that possible cooperations that we may apply to elements of $\mathcal C(\mathbf{t}_i)$ are naturally labelled by rooted trees with leaves labelled by $\mathbf{t}_i$. The cooperation we apply here is the one labelled by the tree $\mathbf{t}'$, with labelled leaves attached according to $f$ and with the labelling of the vertices of $\mathbf{t}'$ disregarded.} 
$$
\mathcal C(\mathbf{t}_i)\longrightarrow
%\mathcal C(J)\otimes
\bigotimes_{j\in J}\mathcal C\big(f^{-1}(j)\cup \mathbf{t}'_j\big)
%\longrightarrow\bigotimes_{j\in J}\mathcal C\big(f^{-1}(j)\big)
\,.
$$
Then observe that we have natural maps 
\[
\mathcal C\big(f^{-1}(j)\cup \mathbf{t}'_j\big)
\otimes \mathcal C\big(\mathbf{t}'_j\big)
\to
\mathcal C\big(f^{-1}(j)\cup \mathbf{t}'_j\big)
\otimes 
\mathcal C\big(f^{-1}(j)\cup \mathbf{t}'_j\big)
\to 
\mathcal C\big(f^{-1}(j)\cup \mathbf{t}'_j\big)
=C\big((\mathbf{t}\circ_i\mathbf{t}')_j\big)
\]
where the first map uses the extension map \eqref{equ:extension} on the second factor and the second map uses the Hopf structure, i.~e., it is the multiplication of the algebra $\mathcal C\big(f^{-1}(j)\cup \mathbf{t}'_j\big)$. 

\medskip

The definition is made such that $\pLie_{\op C}$ naturally acts on the convolution ``algebra''
$$
\Conv_0(\mathcal C\{k\},\mathcal P):=\prod_{n\geq0}\mathsf{Hom}(\mathcal C\{k\},\mathcal P)(n)^{S_n}. 
$$
More generally, for any $\mathcal C$-operad $\mathcal O$, $\prod_n\mathcal O(n)^{S_n}$ is a $\pLie_{\mathcal C}$ algebra 
(and we have already seen that $\mathsf{Hom}(\mathcal C\{k\},\mathcal P)$ is a $\mathcal C$-operad). 

\begin{rem} 
 Note that for $\mathcal C=\uCom^*$ we recover the usual $\pLie$ operad, i.~e., $\pLie_{\uCom^*}=\pLie$. Furthermore the construction $\pLie_{\mathcal C}$ is functorial in $\mathcal C$. Hence from the unit map $\uCom^*\to \mathcal C$ we obtain a map of operads $\pLie \to \pLie_{\mathcal C}$ for any Hopf cooperad $\mathcal C$.
 In particular, any $\pLie_{\mathcal C}$ algebra is a $\Lie$ algebra, and we recover the usual Lie algebra structure on $\Conv_0(\mathcal C\{k\},\mathcal P)$.
\end{rem}

Next, if we have a morphism $f:\Omega(\op C\{k\})\to \mathcal P$ of operads, then it determines a Maurer-Cartan element $\gamma_f$ 
in $\Conv_0(\op C\{k\},\mathcal P)$, and a new convolution dg Lie algebra $\Conv_f(\op C\{k\},\mathcal P)$ is obtained from the original one one by twisting with $\gamma_f$. Note that we may drop the ``$f$'' from the notation when there is no ambiguity. 
In general, the action of the operad $\pLie_{\op C}$ will unfortunately not lift to an action on the twisted convolution Lie algebra $\Conv_f(\op C\{k\},\mathcal P)$. 
However, we now may invoke the formalism of operadic twisting \cite{vastwisting}.
Given an operad $\op P$ together with a map $\Lie\to \op P$, operadic twisting produces:
\begin{itemize}
 \item An operad $\Tw \op P$, the \emph{twisted operad}.
 \item Operad maps $\Lie\to \Tw\op P\to \op P$ whose composition is the given map $\Lie\to \op P$.
 \item The operad $\Tw \op P$ has the property that if we are given a $\op P$ algebra $A$, together with a Maurer-Cartan element $m$ of the $\Lie$ algebra $A$, then the action of the $\Lie$ operad on the twisted $\Lie$ algebra $A^m$ lifts naturally to an action of the operad $\Tw \op P$.
\end{itemize}

In our case we obtain an operad\footnote{Which does NOT depend on $f$.} $\Tw\pLie_{\op C}$, acting naturally on the twisted 
convolution algebra $\Conv_f(\op C\{k\},\mathcal P)$.
Concretely the operad $\Tw\pLie_{\op C}$ is a completed version of the operad generated by $\pLie_{\op C}$ and one formal nullary element.
%corresponding to the Maurer-Cartan element $\gamma_f$. 
The differential on  $\Tw\pLie_{\op C}$ is defined so that upon replacing the formal nullary element by the Maurer-Cartan element $\gamma_f$ we obtain an action of  $\Tw\pLie_{\op C}$ on $\Conv_f(\op C\{k\},\mathcal P)$.
The formal nullary element we denote in pictures by coloring the appropriate vertices of the tree black. We call these vertices the internal vertices (as opposed to external ones). 
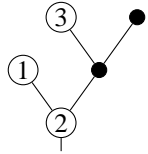
\[
 \begin{tikzpicture}[yscale=-1]
\node (v0) at (0,.5) {};
  \node[ext] (v1) at (0,0) {2};
  \node[ext] (v2) at (-.5,-.7) {1};
  \node[int] (v3) at (.5,-.7) {};
  \node[ext] (v4) at (0,-1.4) {3};
  \node[int] (v5) at (1,-1.4) {};
  \draw (v0) edge (v1) (v1) edge (v2) edge (v3) (v3) edge (v4) edge (v5);   
 \end{tikzpicture}
\]  

Combinatorially, the differential on $\Tw\pLie_{\op C}$ splits vertices, either an internal vertex into two internal vertices, or an external vertex into an external and an internal vertex.

We define the brace construction $\Br(\op C)= \Tw\pLie_{\op C}$ as a synonym for the twisted pre-Lie operad.
By construction $\Br(\ee_n^*)$ acts on the convolution dg Lie algebra 
\[
 C(B)[n] = \Conv\big(\ee_n^*\{n\}, \End_B\big)
\]
for any $\hoe_n$ algebra $B$.
For cosmetic reasons and consistency with the literature we make the following definition.

\begin{defi}
 We define the higher braces operad $\Br_{n+1}$ to be the suboperad
 \[
  \Br_{n+1}\subset \Br(\ee_n^*)\{n\}
 \]
 formed by operations whose underlying trees contain no internal vertices with less than 2 children.
\end{defi}

By very definition, the operad $\Br_{n+1}$ acts naturally on the Hochschild complex $C(B)$.

\begin{ex}
%$\Omega(\op C)$-algebra $V$.
%The brace construction of the associative cooperad is the same as the usual braces operad, up to suspension.
The higher braces operad $\Br_2$ is just the usual braces operad, which acts 
on the Hochschild complex of an associative algebra as usual.
% Note that our definition of the Hochschild complex above differs from the standard one in this case by a degree shift by one unit.
% Concretely, our definition is made such that the Hochschild complex is a Lie algebra, not a $\Lie\{1\}$ algebra.
% \[
% \Br \cong \Br(\Ass^*)\{1\}.
% \]
%Taking $\op C=\Ass_1^{\textrm{!`}}$ one obtains the usual braces operad.
\end{ex}

%Next one assumes that there is an operad map $\hoLie_n\to\Br_0(\op C)$ and performs operadic twisting to obtain an operad $\Br(\op C)$ with non-trivial differential. It is defined so as to act on the complex $\Def(\Omega(\op C)\to \End(V))$.
% In general, we will define the higher braces operads as 
% \[
% \Br_{n+1} := \Br(\e_n^{*})\{n\}.
% \] 
% 
% We saw above that $\Br(\e_n^{*})$ acts (in particular) on $\Conv_f(\mathcal \e_n^*\{n-1\},\End)$
% %for $\op C=e_n^{\textrm{!`}}$ one obtains an operad $\Br_{n+1}:=\Br(e_n^{\textrm{!`}})$ that acts naturally on the deformation complex
% \textcolor{red}{By construction, we see that the operad $\Br_{n+1}$ acts on the degree shifted deformation complex 
% \[
%  \Def(\hoe_n\to \End(A))[-n]
% \]
% for any $\hoe_n$ algebra $A$.
% In fact, one may check that the action extends to the Hochschild complex
% \[
% C(A) = A \oplus \Def(\hoe_n\to \End(A))[-n].
% \]}

%{\bf Task/Conjecture:} Show that $H(\Br_{n+1})=e_{n+1}$.

\subsection{Tamarkin's morphism}

D. Tamarkin proved Theorem \ref{thm:enen1} by noting that for $n\geq 2$ there is a quite simple but very remarkable explicit map 
\[
 T\colon \hoe_{n+1}\to \Br_{n+1}.
\]
It is defined on generators by the following prescription:
\begin{itemize}
 \item Generators of the form $\underline{X_1\cdots X_k}\in e_{n+1}^{\textrm{!`}}(k)$ are mapped to a corolla of the form 
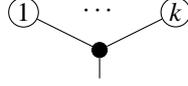
\begin{equation}
\label{equ:blackcorolla}
\begin{tikzpicture}
\node[ext] (v2) at (-1,0.5) {$1$};
\node[ext] (v3) at (1,0.5) {$k$};
\node (v0) at (0,-0.5) {};
\node (vv) at (0,0.5) {$\cdots$};
\node[int] (v1) at (0,0) {};
\draw (v0) edge (v1);
\draw (v2) edge (v1);
\draw (v3) edge (v1);
%\draw (v2) edge[dotted] (v3);
\end{tikzpicture}
\end{equation}

decorated by $\underline{X_1\cdots X_k}\in e_n^{\textrm{!`}}(k)$
\item Generators of the form $X_0\wedge \underline{X_1\cdots X_k}\in e_{n+1}^{\textrm{!`}}(k+1)$ are mapped to a corolla of the form 
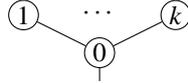
\begin{equation}\label{equ:corwhite}
\begin{tikzpicture}
\node[ext] (v2) at (-1,0.5) {$1$};
\node[ext] (v3) at (1,0.5) {$k$};
\node (v0) at (0,-0.5) {};
\node (vv) at (0,0.5) {$\cdots$};
\node[ext] (v1) at (0,0) {$0$};
\draw (v0) edge (v1);
\draw (v2) edge (v1);
\draw (v3) edge (v1);
%\draw (v2) edge[dotted] (v3);
\end{tikzpicture}
\end{equation}
decorated by $\underline{X_1\cdots X_k}\in e_n^{\textrm{!`}}(k)$. In the special case $k=1$ one takes the (anti-)symmetric combination of the two possible choices.
\item All other generators are mapped to zero.
\end{itemize}
This prescription indeed gives a morphism $\hoe_{n+1}\to \Br_{n+1}$ of the underlying graded operads (because, as such, $\hoe_{n+1}$ is free).
This leaves us with the task of verifying that the map $T:\hoe_{n+1}\to \Br_{n+1}$ commutes with the differentials.
It suffices to check this on the generators.
Furthermore it suffices to check the statement on generators of one of the forms
\begin{align}\label{equ:threegenerators}
 &\underline{X_1\cdots X_k}
 & &\underline{X_1\cdots X_k}\wedge \underline{X_{k+1}\cdots X_{k+l}}
 & &X_1\wedge \underline{X_{2}\cdots X_{k+1}}\wedge \underline{X_{k+2}\cdots X_{k+l+1}}
\end{align}
since in all other cases the differential of the generator and the generator itself are mapped to zero, so that the map $T$ trivially commutes with the differentials.
One has to check each of the three types of generators above in turn. 
The calculation is a bit lengthy, due to several special cases that need to be considered.
%Since the construction of $T$ is essentially the result of D. Tamarkin \cite{tamenaction} that we just recall here, we will be brief and not produce the whole calculation.
Since the construction of $T$ is essentially the result of D. Tamarkin \cite{tamenaction} we will only show how to handle a few cases in Appendix \ref{app} as an illustration.  

\section{$\Br_n$ is an $E_n$ operad}
\begin{thm}\label{thm:HBrn}
 The above map $T:\hoe_{n+1}\to \Br_{n+1}$ is a quasi-isomorphism of operads for all $n=2,3,4,\dots$, so in particular $H(\Br_{n+1})\cong \e_{n+1}$.
\end{thm}
In the case $n=1$ it is still true that there is a quasi-isomorphism $\hoe_{2}\to \Br_{2}$, but this morphism is much more complicated to construct than the Tamarkin quasi-isomorphism $T$ we described above. It can be obtained by combining a quasi-isomorphism from $\Br_2$ to the chains of the little disks operad \cite{KS1} with a choice of formality morphism of the little disks operad.

Theorem \ref{thm:HBrn} is not used in this note, so we only sketch the proof.
\begin{proof}[Sketch of proof]
 First one checks that $H(\Br_{n+1})\cong \e_{n+1}$. The proof of this statement follows along the lines of the proof of the $n=1$ case in \cite[Appendix C]{vastwisting}.
 The only point where the proof in loc. cit. has to be adapted is that in \cite[Appendix C.2]{vastwisting} one has to compute the Hochschild cohomology of a free $\e_{n+1}$ algebra, considered as an $\e_{n}$ algebra, instead of computing the Hochschild cohomology of a free $\e_2$ algebra, considered as an $\e_{1}$ algebra. The answer is provided by the higher Hochschild-Kostant-Rosenberg Theorem, i.~e. Theorem \ref{thm:higherHKR} above, instead of the usual one.
 
 Once one knows that $H(\Br_{n+1})\cong \e_{n+1}$, the statement of the Theorem is shown by checking that the induced map in cohomology
 \[
  \e_{n+1}\cong H(\hoe_{n+1})\to H(\Br_{n+1})\cong \e_{n+1}
 \]
is the identity, which amounts to checking that it is the identity on the two generators.
\end{proof}

The brace construction $\Br_{n+1}$ is intuitively similar to taking a product of $\Omega(\e_n^{\textrm{!`}})$ with an $E_1$ operad. So the above Theorem shall be understood as a version of the statement that the product of an $E_1$ operad with an $E_n$ operad is an $E_{n+1}$ operad.

\section{Proof of Theorem \ref{thm:main}}
\label{sec:theproof}
 One needs to verify that the action of all generators of $\hoe_{n+1}$ commutes with the map. On the left hand side all actions are zero except for that of the generators $X_1\wedge X_2$ (i.e., the Lie bracket) and of $\underline{X_1X_2}$ (i.e., the product). The fact that the HKR morphism respects the Lie bracket is a simple verification. The fact that the product is preserved is obvious. Hence it suffices to check that the action of the generators $\underline{X_1\cdots X_k}$ ($k\geq 3$) and $X_0\wedge \underline{X_1\cdots X_k}$ ($k\geq 2$) on the image of the HKR map is trivial.
The generators $\underline{X_1\cdots X_k}$ act using the corresponding components of the $\hoe_{n}$ structure on $A_n$, which vanish. Hence they act trivially (as long as $k\geq 3$). Note also that the image of the HKR map has only ``$\hoLie_{n}$ components'', i.e., the corresponding maps $\e_{n}^{\textrm{!`}}(N)\to \End(V)(N)$ factor through $\e_{n}^{\textrm{!`}}(N)\to \Lie_{n}^{\textrm{!`}}(N)$. However, the prescription for the action of the component $X_0\wedge \underline{X_1\cdots X_k}$ advises us to evaluate the arguments on components $\underline{X_1\cdots X_k}$, which are sent to zero under the projection $\e_{n}^{\textrm{!`}}(k)\to \Lie_{n}^{\textrm{!`}}(k)$. Hence the action of the components $X_0\wedge \underline{X_1\cdots X_k}$ vanishes (as long as $k\geq 2$) on the image of the HKR map.
\hfil \qed

\appendix

\section{The map $T$ commutes with the differentials}\label{app}

\subsection{The generator $\underline{X_1\cdots X_k}$}

In this case the differential of the generator consists of 
$$
\sum_{j,r}\pm\underline{X_1\dots X_{j-1}*X_{j+r+1}\dots X_k}\circ_*\underline{X_j\dots X_{j+r}}\,,
$$
where the notation $A\circ_* B$ shall mean the operadic composition in $\hoe_{n+1}$ of the operations $A$ and $B$ in $\hoe_{n+1}$, 
with $B$ being ``inserted in the slot'' of $A$ labelled by $*$. The map $T$ sends the above to 
$$
\sum_{j,r}\pm
\begin{tikzpicture}[baseline=-0.65ex]
\node[ext] (v2) at (-1,0.5) {$1$};
\node[ext] (v3) at (1,0.5) {$k$};
\node (v0) at (0,-0.5) {};
\node (vv) at (-0.5,0.5) {$\scriptstyle\cdots$};
\node (vv) at (0.5,0.5) {$\scriptstyle\cdots$};
\node[int] (vi) at (0,0.5) {};
\node[int] (v1) at (0,0) {};
\node[ext] (w1) at (-.5,1) {$j$};
\node (wd) at (0,1) {$\scriptstyle\cdots$};
\node[ext] (w2) at (.5,1) {$\scriptstyle j+r$};
\draw (v0) edge (v1);
\draw (v2) edge (v1);
\draw (v3) edge (v1) (vi) edge (w1) edge (w2) (v1) edge (vi);
\end{tikzpicture}
$$
This is precisely the differential of the tree \eqref{equ:blackcorolla}, which is the image of $\underline{X_1\cdots X_k}$ by $T$. Decorations are obvious. 

\subsection{The generator $X_0\wedge\underline{X_1\cdots X_k}$}

In this case the differential of the generator consists of the following types of 
  \begin{multline}\label{equ:dcontributions}
   \sum_{j}\sum_{r\geq1} \pm X_0 \wedge \underline{X_1\cdots X_{j-1}*X_{j+r+1}\cdots X_{k}}
   \circ_* \underline{X_j\cdots X_{j+r}}
   +
   \sum_{j}\sum_{r\geq1} \pm \underline{X_1\cdots X_{j-1}*X_{j+r+1}\cdots X_{k}}
   \circ_* (X_0 \wedge \underline{X_j\cdots X_{j+r}})
   \\
   \pm (X_0\wedge *) \circ_* \underline{X_1\cdots X_{k}}
   +\sum_j \pm \underline{X_1\cdots X_{j-1} * X_{j+2}\cdots X_{k}}\circ_* (X_0\wedge X_j)\,.   
  \end{multline}

This is mapped under $T$ to a linear combination of trees of the following form 
\[
\sum_j \sum_{r\geq 1} \pm
 \begin{tikzpicture}[baseline=-0.65ex]
\node[ext] (v2) at (-1,0.5) {$1$};
\node[ext] (v3) at (1,0.5) {$k$};
\node (v0) at (0,-0.5) {};
\node (vv) at (-0.5,0.5) {$\scriptstyle\cdots$};
\node (vv) at (0.5,0.5) {$\scriptstyle\cdots$};
\node[int] (vi) at (0,0.5) {};
\node[ext] (v1) at (0,0) {$0$};
\node[ext] (w1) at (-.5,1) {$j$};
\node (wd) at (0,1) {$\scriptstyle\cdots$};
\node[ext] (w2) at (.5,1) {$\scriptstyle j+r$};
\draw (v0) edge (v1);
\draw (v2) edge (v1);
\draw (v3) edge (v1) (vi) edge (w1) edge (w2) (v1) edge (vi);
%\draw (v2) edge[dotted] (v3);
\end{tikzpicture}
+\sum_j \sum_{r\geq 0} \pm
 \begin{tikzpicture}[baseline=-0.65ex]
\node[ext] (v2) at (-1,0.5) {$1$};
\node[ext] (v3) at (1,0.5) {$k$};
\node (v0) at (0,-0.5) {};
\node (vv) at (-0.5,0.5) {$\scriptstyle\cdots$};
\node (vv) at (0.5,0.5) {$\scriptstyle\cdots$};
\node[ext] (vi) at (0,0.5) {$0$};
\node[int] (v1) at (0,0) {};
\node[ext] (w1) at (-.5,1) {$j$};
\node (wd) at (0,1) {$\scriptstyle\cdots$};
\node[ext] (w2) at (.5,1) {$\scriptstyle j+r$};
\draw (v0) edge (v1);
\draw (v2) edge (v1);
\draw (v3) edge (v1) (vi) edge (w1) edge (w2) (v1) edge (vi);
%\draw (v2) edge[dotted] (v3);
\end{tikzpicture}
\pm 
\begin{tikzpicture}[baseline=-0.65ex]
\node[ext] (v11) at (-1.5,0.5) {$0$};
\node[ext] (v2) at (-.7,0.5) {$1$};
\node[ext] (v3) at (1,0.5) {$k$};
\node (v0) at (0,-0.5) {};
\node (vv) at (0,0.5) {$\cdots$};
\node[int] (v1) at (0,0) {};
\draw (v0) edge (v1);
\draw (v2) edge (v1);
\draw (v3) edge (v1);
\draw (v1) edge (v11);
%\draw (v2) edge[dotted] (v3);
\end{tikzpicture}
\]
Here all corollas are decorated by the top degree elements of $\ee_n^*$, except for the last tree, where the decoration is by the element
$X_0\wedge \underline{X_{1}\cdots X_{k}}$. 

One checks that this linear combination of trees is exactly the differential of \eqref{equ:corwhite}, which is the image of the generator we considered by $T$.
Note also that the trees of the form
\[
  \begin{tikzpicture}[baseline=-0.65ex]
\node[ext] (v2) at (-1,0.5) {$1$};
\node[ext] (v3) at (1,0.5) {$k$};
\node (v0) at (0,-0.5) {};
\node (vv) at (-0.5,0.5) {$\scriptstyle \cdots$};
\node (vv) at (0.5,0.5) {$\scriptstyle \cdots$};
\node[ext] (vi) at (0,0.5) {$j$};
\node[int] (v1) at (0,0) {};
\node[ext] (w1) at (0,1.1) {$0$};
\draw (v0) edge (v1);
\draw (v2) edge (v1);
\draw (v3) edge (v1) (vi) edge (w1) (v1) edge (vi);
%\draw (v2) edge[dotted] (v3);
\end{tikzpicture}
\]
occur twice, with the two contributions from the third and fourth term of \eqref{equ:dcontributions} cancelling each other.

\end{document}